\documentclass[11pt]{amsart}

\usepackage{amsfonts,amsmath,amsthm}

\usepackage[english]{babel}
\usepackage{color}

\theoremstyle{plain}
\newtheorem{thm}{Theorem}[section]
\newtheorem{lem}[thm]{Lemma}
\newtheorem{prop}[thm]{Proposition}
\newtheorem{cor}[thm]{Corollary}
\newtheorem{remark}[thm]{Remark}

\numberwithin{equation}{section}

\definecolor{tl}{rgb}{0.7,0.1,0.2}

\def \NN {\mathbb{N}}
\def \RR {\mathbb{R}}

\def \BB {\mathbb{B}}
\def \s {\mathbb{S}}

\begin{document}

\title[On the Green function and Poisson integrals]{On the Green function and Poisson integrals of the Dunkl Laplacian}


\author{Piotr Graczyk}
\address{Piotr Graczyk, LAREMA, Universit\'e d'Angers, 2 Bd Lavoisier, 49045 Angers Cedex 1, France}
\email{graczyk@math.univ-angers.fr}

\author{Tomasz Luks}
\address{Tomasz Luks, Institut f\"ur Mathematik, Universit\"at Paderborn, Warburger Strasse 100, D-33098 Paderborn, Germany}
\email{tluks@math.uni-paderborn.de}

\author{Margit R\"osler}
\address{Margit R\"osler, Institut f\"ur Mathematik, Universit\"at Paderborn, Warburger Strasse 100, D-33098 Paderborn, Germany}
\email{roesler@math.upb.de}

\maketitle
\selectlanguage{english}

\begin{abstract}
We prove the existence and study properties of the Green function of the unit ball for the Dunkl Laplacian $\Delta_k$ in $\RR^d$. As applications we derive the Poisson-Jensen formula for $\Delta_k$-subharmonic functions and Hardy-Stein identities for the Poisson integrals of $\Delta_k$. We also obtain sharp estimates of the Newton potential kernel, Green function and Poisson kernel in the rank one case in $\RR^d$. These estimates contrast sharply with the well-known results in the potential theory of the classical Laplacian.
\end{abstract}
\selectlanguage{english}

\section{Introduction}

Dunkl operators are differential reflection operators associated with finite reflection groups which generalize the usual partial derivatives as well as the invariant differential operators of Riemannian symmetric spaces. They play an  important role in harmonic analysis and the study of special functions of several variables. Among other applications, Dunkl operators are employed in the description of quantum integrable models of Calogero-Moser type, see e.g. \cite{vDV}. Also, there are stochastic processes associated with  Dunkl Laplacians which generalize Dyson's Brownian motion model, see e.g. \cite{GY, VR}. Recently, the potential theory of the Dunkl Laplacian $\Delta_k$ has found increasing attention in view of many interesting open problems and the need of developing new techniques, as many standard methods known from the case of diffusion operators do not apply, see, e.g., \cite{GaRe1,MaYo1,Re}. In the present paper we study the properties of one of the fundamental objects in the potential theory of $\Delta_k$: the Green function $G_k(x,y)$ of the unit ball $\BB$ in $\RR^d$. The behavior and estimates of this function and its generalizations for bounded smooth domains were intensively studied in the case of the classical Laplacian \cite{Bog,Wid,Z1,Z2}, more general diffusion operators \cite{An,Ar,CrZ,GrWi,KiSo,LSW}, as well as nonlocal operators \cite{BJ,CheS,GrzR,J,Ku}.

Our first result, Theorem \ref{thm:GreenBall}, establishes the existence and an integral formula for $G_k(x,y)$. A more convenient two-sided bound of $G_k(x,y)$ is given in Theorem \ref{thm:GreenEst}. We also prove a standard relation between $G_k(x,y)$ and the Poisson kernel $P_k(x,y)$ of $\BB$ for $\Delta_k$, see Proposition~\ref{thm:GreenPoisson}. As applications of Theorem~\ref{thm:GreenBall} we obtain the Poisson-Jensen formula for $\Delta_k$-subharmonic functions and Hardy-Stein identities for $\Delta_k$-harmonic functions on $\BB$, see Theorem~\ref{thm:JensPois} and Theorem~\ref{thm:HardyStein}. This leads to an equivalent characterization of the Hardy spaces of $\Delta_k$ on $\BB$ in the spirit of \cite{BDL}. We remark that the general integral representation \eqref{GreenBallDiff} of $G_k(x,y)$ and the estimate of Theorem \ref{thm:GreenEst} involve the representing measure for the intertwining operator whose structure depends strongly on the underlying root system. Note that explicit formulas for the representing measure are known only in a few particular cases, and the question whether it always admits a Lebesgue density is a challenging open problem. However, the available results together with Theorem \ref{thm:GreenEst} allow us to 
derive explicit two-sided bounds of the Newton kernel $N_k(x,y)$, the Green function $G_k(x,y)$ 
and the Poisson kernel $P_k(x,y)$ for $\Delta_k$ in the rank one case in $\RR^d$, see Theorem \ref{thm:NewtonEst}, Theorem \ref{thm:GreenEstR1}, and Corollary \ref{thm:Poisson}. The obtained estimates contrast sharply with the classical results in the potential theory of the Laplacian 
$\Delta$ or more general diffusion operators. The main novelties in the present setting are additional singularities of $N_k(x,y)$ and $G_k(x,y)$ in $x=gy$ in dimensions higher than 3 ($g$ is in the associated reflection group $W$) and the dependence of the estimate of $N_k(x,y)$, $G_k(x,y)$ and $P_k(x,y)$ on the distance to the boundary of the Weyl chamber. This makes the obtained asymptotics more complex than in the case of diffusion operators, in particular these for the Green function $G_k(x,y)$. Deriving analogous two-sided bounds in the setting of any other root system is an interesting open problem, and available informations about the representing measure for the intertwining operator are in this case essential. We should note that the existence of singularities of the Newton kernel $N_k(\cdot,y)$ on the orbit $W.y$ has recently been discussed in the case of an orthogonal root system, see \cite[Proposition 2.59]{Re}. 

The paper is organized as follows. In Section 2 we give basic definitions and list some useful facts in the theory of Dunkl operators. In Section 3 we prove the existence and study properties of $G_k(x,y$). In Section 4 we prove the Poisson-Jensen formula and Hardy-Stein identites. In Section 5 we derive sharp estimates of $N_k(x,y)$, $G_k(x,y)$ and $P_k(x,y)$ in the rank one case in $\RR^d$.

\section{Preliminaries}

For details on the following, see \cite{Du0}, \cite{Du1}, \cite{Ro} and, for a general overview, \cite{DuX} or \cite{Ro3}. 
Let $R$ be a  root system in $\mathbb R^d$ (equipped with the usual scalar product and Euclidean norm $|\cdot|$), and let $W$ be the associated finite reflection
group. The root system $R$ needs not be crystallographic and  $W$ is not required to be effective, i.e.
$\text{span}_\mathbb R R$ may be a proper subspace of $\mathbb R^d$. The dimension of 
$\text{span}_\mathbb R R$ is called the rank of $R$.  An important example is $R=A_{d-1} = \{ \pm(e_i-e_j): 1\leq i < j \leq d\}\subset \mathbb R^d$ with $W=S_d$, the symmetric group in $d$ elements. 
We  fix a nonnegative multiplicity function $k$ on $R$, i.e. $k: R\to [0, \infty)$ is $W$-invariant. 
The (rational) Dunkl operators associated with $R$ and $k$ are given by
\[ T_\xi f(x)  = \partial_\xi f(x)  + \sum_{\alpha\in R_+} 
   k(\alpha)\,\langle\alpha, \xi\rangle\, 
    \frac{f(x) - f(\sigma_\alpha x)}{\langle\alpha, x\rangle},
 \quad\xi\in \mathbb R^d,\] 
where $R_+$ denotes an (arbitrary) positive subsystem of $R$. For fixed $R$ and $k$, these operators commute. Moreover, there is a unique linear isomorphism $V_k$ on the space of polynomial functions in $d$ variables, called the intertwining operator, which preserves the degree of homogeneity, is normalized by $V_k(1)=1$  and intertwines  the Dunkl operators with the usual partial derivatives:
\[ T_\xi V_k = V_k \partial_\xi \quad \text{ for all } \xi \in \mathbb R^d.\]
The Dunkl Laplacian is defined by 
$$ \Delta_k := \sum_{i=1}^ d T_{\xi_i}^2 $$
with an (arbitrary) orthonormal basis $(\xi_i)_{ 1\leq i \leq d}$ of $\mathbb R^d$.  In explicit form,
$$
\Delta_k f(x)=\Delta f(x)+\sum_{\alpha\in R_+}k(\alpha)\left(\frac{\langle\nabla f(x),\alpha\rangle}{\langle\alpha,x\rangle}-\frac{f(x)-f(\sigma_\alpha(x))}{\langle\alpha,x\rangle^2}\right),
$$
where $\Delta$ is the usual Laplacian on $\mathbb R^d$. 
For $x\in \RR^d$   denote  by $C(x)$ the convex hull of the Weyl group
orbit $W.x$ of $x$ in $\RR^d.$
The intertwining operator $V_k$ has the integral representation
\begin{equation}\label{intertwiner}
V_kf(x)=\int_{C(x)} f(z) d\mu_x^k(z),
\end{equation}
where $\mu_x^k$ is a probability measure on $C(x).$ 
The measures $\mu_x^k$ satisfy 
\[\mu_{rx}^k(A) = \mu_x^k(r^{-1}A) \]
for all $r>0$ and Borel sets $A\subseteq\mathbb R^d$. 
In \cite{T}, it was  
deduced from formula \eqref{intertwiner} that $V_k$ 
 establishes a homeomorphism of $C^\infty(\mathbb R^d)$ with its usual 
Fr\'{e}chet 
space topology.

In the rank one case $R = \{\pm 1\}\subset \mathbb R$, the representation \eqref{intertwiner} is explicitly known 
(\cite[Theorem 5.1]{Du1}); it is given by
\begin{equation}\label{V_dimension_one} 
V_kf(x) = c_k \int_{-1}^1 f(tx)(1-t)^{k-1}(1+t)^k\, dt\quad \text{ with } 
c_k = \frac{ \Gamma(k+1/2)}{\sqrt{\pi}\,\Gamma(k)}. \end{equation}

We shall employ the Dunkl-type generalized translation on $C^\infty(\mathbb R^d)$
which was defined in \cite{T} by
\[ \tau_yf(x) := V_k^xV_k^y (V_k^{-1}f)(x+y),\,\, x,y \in \mathbb R^d.\]
Here the superscript denotes the relevant variable. This translation satisfies 
$\,\tau_yf(x) = \tau_xf(y),$ and we shall use the notation
$f(x*_k y):= \tau_yf(x).$

\begin{lem} \begin{enumerate}\itemsep=-1pt
               \item[(i)] 
The representing measures $\mu_x^k$ satisfy
$\mu_{-x}^k(-A) = \mu_x^k(A).$ 
 \item[(ii)] Let $f\in C^\infty(\mathbb R^d)$ and write $f^-(x):= f(-x).$ Then 
 $f(-x*_k -y) = f^-(x*_ky).$
 \end{enumerate}
\end{lem}

\begin{proof}
 It is immediate that the Dunkl operators satisfy $T_\xi(f^-) = (T_{-\xi}f)^-\,.$ 
 By the characterization  of $V_k$, it follows that $V_k(f^-) = (V_k f)^-\,.$  This implies both assertions.
\end{proof}

Of particular importance in our context will be translates of functions  $f$ on $\mathbb R^d$ which are radial,
that is $f(x)= \widetilde f(|x|)$ with $\widetilde f: [0,\infty) \to \mathbb C$. 
We recall from \cite{Ro2} that for each $x,y\in \mathbb R^d$ there exists a unique compactly supported
radial probability measure $\rho_{x,y}^k$ on $\mathbb R^d$ such that 
\begin{equation}\label{radialintegral} f(x*_ky) = \int_{\mathbb R^d} fd\rho_{x,y}^k\end{equation}
for all  $f\in C^\infty(\mathbb R^d).$
This can be written explicitly as
\begin{equation}\label{transexplicit} f(x*_ky) = 
\int_{C(y)}\widetilde f\bigl(\sqrt{|x|^2 + |y|^2 + 2 \langle x,z\rangle}\,\bigr) d\mu_y^k(z).\end{equation}
Notice that Dunkl translates of non-negative, smooth  radial functions are again non-negative.
Formula \eqref{radialintegral} allows to extend the generalized translation to measurable radial functions which are
either complex-valued and bounded or have values in $[0, \infty].$ 
We maintain the notations $\tau_yf(x)$ and $f(x*_k y)$  for functions from these classes.
In particular, for radial $f$  we have 
\begin{equation}\label{radialminus}
f(-x *_k -y ) = f(x*_ky).
\end{equation}
We put 
\[\gamma := \sum_{\alpha\in R_+}k_\alpha\,\]
and define the weight function $\omega_k$ on $\mathbb R^d$ by
$$
\omega_k(x):=\prod_{\alpha\in R_{+}}\left|\langle\alpha,x\rangle\right|^{2k_{\alpha}}.
$$ 

Let $\BB = \{ x\in \mathbb R^d: |x| < 1\}$ denote the open unit ball 
in $\mathbb R^d$ and let $\s=\partial\BB$ denote the unit sphere.
The Poisson kernel $P_k(x,y)$  of  $\BB$  for the Dunkl
Laplacian $\Delta_k$ was defined in \cite{Du1} as a reproducing kernel
for $\Delta_k$-harmonic polynomials. It can be written as
\begin{equation}\label{PoissonDunkl}
P_k(x,y)=V_k\left[\frac{1-|x|^2}{(1-2\langle x,\cdot\rangle +|x|^2)^{\gamma+
d/2}}\right] (y),\,\, x\in\BB, y\in\s.
\end{equation}
In view of identity \eqref{radialminus} with $f(x) = |x|^{-2\gamma -d}$, we obtain
\begin{align}\label{PoissonRoesler}
P_k(x,y)&=\,\int_{C(y)} \frac{1-|x|^2}{(1-2\langle x, z\rangle
+|x|^2)^{\gamma+ d/2}}  d\mu_y(z)\,=\, (1-|x|^2) \cdot f(-x*_ky)\notag \\ 
&=  \,(1-|x|^2) \cdot\tau_{-y}(|x|^{-2\gamma-d}).
\end{align}


The notation $f\asymp g$ will always mean that there is a constant $C>0$
depending on $k$ and $d$ only (unless stated otherwise) such that $C^{-1}g\leq f \leq Cg$.

\section{The Green function of the ball}

From now on, it is always assumed that $d+2\gamma>2$. Following \cite{ Re}, we introduce the Newton kernel in the Dunkl setting by
\[ N_k(x,y) = \int_0^\infty \Gamma_k(t,x,y)dt\quad (x,y\in \mathbb R^ d),\]
with the heat kernel
\[ \Gamma_k (t,x,y) =
\frac{M_k}{t^{\gamma+ d/2}}\,e^{-(|x|^2+|y|^2)/4t} 
E_k\bigl(\frac{x}{\sqrt{2t}}, \frac{y}{\sqrt{2t}}\bigr),\]
where 
\[M_k = 2^{-\gamma-d/2}\bigl(\int_{\RR^d} e^{-|x|^2/2}\omega_k(x)dx\bigr)^{-1}.\]
Notice that $\, N_k(x,y) = N_k(y,x).$ 
According to the results in \cite[Section 2.7]{Re}, the  Newton kernel can be written as 
\begin{equation}\label{DunklNewton}
N_k(x,y)=C_k\int_{C(y)}\left(|x|^2+|y|^2-2\langle x, z\rangle\right)^{1-\gamma-d/2}d\mu_y(z) 
\end{equation} 
where 
\begin{equation}\label{constant_C_k}
C_k = \frac{1}{d_k(d+2\gamma-2)} \quad\text{and }\, d_k = \int_{\s} \omega_k(x)d\sigma(x).
\end{equation}
Here $\sigma$ denotes the surface measure on $\s$. 
Formula \eqref{DunklNewton} is also easily obtained by translations. Recall that 
\[ \Gamma_k(t,x,y) = \tau_{-y}g_t(x) \quad \text{with } \,g_t(x) = \frac{M_k}{t^{\gamma+ d/2}}\,e^{-|x|^2/4t},\] 
which follows from \cite[Lemma 2.2. and (3.2)]{Ro2} (see also \cite{VR}).
As 
\[ \int_0^\infty g_t(\xi)dt \,=\,M_k\Gamma (\gamma + \frac{d}{2}-1) \cdot\bigl(\frac{2}{|\xi|}\bigr)^{d-2+2\gamma}\,=\,\frac{C_k}{|\xi|^{d-2+2\gamma}},\]
it follows that
\[ N_k(x,y) =  \int_0^\infty  \int_{\mathbb R^d} g_t(\xi) d\rho_{x,-y}^k(\xi)   dt \,= \, 
 C_k\cdot
\tau_{-y}\bigl(|x|^{2-2\gamma-d}\bigr).\]
In view of identity \eqref{radialminus}, this equals the right-hand side of \eqref{DunklNewton}. Furthermore, the Newton kernel 
 $N_k(\cdot,y)$ is $\Delta_k$-harmonic on $\RR^d\setminus W.y$ for fixed $y\in \mathbb R^ d$ (see \cite[Proposition~2.64]{Re}).
It can  be regarded as the global Green function for the Dunkl Laplacian $\Delta_k$.

The goal of this section is to  introduce and study the Green function of the ball $\BB$ for $\Delta_k$. For this, we  recall from \cite{KaYa} 
the Kelvin transform associated 
with the Dunkl Laplacian, which is given by
$$
K_k[u](x)=|x|^{2-2\gamma-d}u(x^*) 
$$
for functions $u$ on $\RR^d\setminus\{0\}$, where $x^*=x/|x|^2$ is the inversion with respect to the unit sphere in $\RR^d$. 
By \cite[Theorem 3.1]{KaYa}, $K_k$ preserves $\Delta_k$-harmonic functions on $\RR^d\setminus\left\{0\right\}$.
Following the classical case  $k=0$ (cf. \cite{Do, St}),  we define \begin{equation}\label{GreenDef}
G_k(x,y):=N_k(x,y)-K_k[N_k(\cdot,y)](x)
\end{equation}
for $x,y \in \overline{\BB}\times\overline{\BB}$ with $x\not=0$, where $K_k[N_k(\cdot,y)](x)=|x|^{2-2\gamma-d}N_k(x^*,y)$. 

\begin{thm}\label{thm:GreenBall} The kernel $G_k$ is the Green function of $\BB$ for $\Delta_k$, that is, $G_k$  
extends to a $[0,\infty]$-valued function on $\overline{\BB}\times\overline{\BB}$ which is uniquely characterized by the following conditions:
\begin{enumerate}
\item[(i)] $G_k(x,y)>0$ for all $x,y\in\BB$ and $G_k(x,y)=0$ for $x\in\s$ and $y\in\BB$.
\item[(ii)] $G_k(\cdot,y)$ is continuous on $\overline\BB\setminus W.y$ for any fixed $y\in\BB$.
\item[(iii)] $N_k(\cdot,y)-G_k(\cdot,y)$ is  $\Delta_k$-harmonic on $\BB$ for any fixed $y\in\BB$.
\end{enumerate}
Moreover, the Green function $G_k$ can be written as 
\begin{align}\label{GreenBallDiff}
G_k(x,y)=C_k\int_{C(y)}&\left[\left(|x|^2+|y|^2-2\langle x, z\rangle\right)^{1-\gamma-d/2}\right.\\
&\left.-\left(1+|x|^2|y|^2-2\langle x, z\rangle\right)^{1-\gamma-d/2}\right]d\mu_y(z).\nonumber
\end{align}
It satisfies $G_k(x,y)=G_k(y,x)$ for all $x,y\in\overline{\BB}$, and $G_k(\cdot,y)$ is 
$\Delta_k$-harmonic on $\BB\setminus W.y$ for any fixed $y\in\BB$.
\end{thm}

\begin{proof}
As $N_k(\cdot,y)$ is $\Delta_k$-harmonic on $\RR^d\setminus W.y$ for fixed $y\in\RR^d$, its Kelvin transform 
$K_k[N_k(\cdot,y)]$ is $\Delta_k$-harmonic on $\BB\setminus\left\{0\right\}$ and continuous on $\overline{\BB}\setminus\left\{0\right\}$ for any fixed $y\in\BB$. 
By \eqref{DunklNewton} we have
\begin{align}\label{K_k_integralformula}
K_k[N_k(\cdot,y)](x)=C_k\int_{C(y)}\left(1+|x|^2|y|^2-2\langle x, z\rangle\right)^{1-\gamma-d/2}d\mu_y(z),
\end{align}
and from this representation it is immediate by the dominated convergence theorem that $K_k[N_k(\cdot,y)]$ has a removable singularity at 0 for  fixed $y~\in~\BB$. 
Employing \cite[Theorem 5.1]{GaRe1}, we conclude that $K_k[N_k(\cdot,y)]$ extends to a $\Delta_k$-harmonic function on $\BB$. Furthermore, $K_k[N_k(\cdot,y)]$ 
solves the $\Delta_k$-Dirichlet problem on $\overline{\BB}$ with the boundary values of $N_k(\cdot,y)$.
Therefore, $G_k(\cdot,y)$ vanishes continuously at $\s$ and is 
$\Delta_k$-harmonic on $\BB\setminus W.y$. Formula \eqref{K_k_integralformula} immediately gives the claimed identity \eqref{GreenBallDiff}.
As $\,1+|x|^2|y|^2 -2\langle x,z\rangle > |x|^2+|y|^2-2\langle x,z\rangle \geq 0\,$ for all $x,y \in \BB$ and $z\in C(y)$,  it follows from \eqref{GreenBallDiff} that 
$G_k(x,y) >0$ for all $x,y\in  \BB.$
For the symmetry of $G_k$, it suffices to prove that 
$K_k[N_k(\cdot,y)](x)$ is symmetric in $x$ and $y$ for $x\not=0.$ Using the symmetry of $N_k$ and the fact that 
for any $r>0$, the representing measure $\mu_{rx}$ is just the image measure of $\mu_x$ under the dilation $z\mapsto rz$ of $\mathbb R^d$,  we obtain
\begin{align*} K_k[N_k(\cdot,y)](x) &= |x|^{2-2\gamma-d}N_k(y,x^*) \\
   & = |x|^{2-2\gamma-d}\cdot C_k\int_{\mathbb R^d}(|y|^2 + |x^*|^2 - 2\langle y, z\rangle)^{1-\gamma-d/2}d\mu_{x/|x|^2}(z) \\
   & = |x|^{2-2\gamma-d}\cdot C_k\int_{\mathbb R^d}(|y|^2 + |x^*|^2 - 2\langle y, \frac{z}{|x|^2}\rangle)^{1-\gamma-d/2}d\mu_{x}(z) \\
   & = C_k\int_{\mathbb R^d}(|x|^2|y|^2+1-2\langle y,z\rangle)^{1-\gamma-d/2}d\mu_{x}(z) \\
   & = K_k[N_k(\cdot,x)](y).
\end{align*}
Further, \cite[Proposition 2.58]{Re} gives $G_k(x,x)=+\infty$. Finally, the uniqueness of the function $G_k$ subject to the conditions $(i)-(iii)$
 follows from the uniqueness
of solutions to the $\Delta_k$-Dirichlet problem on $\BB$, see \cite{MaYo1}.
\end{proof}

According to \cite[Proposition 2.64]{Re}, $-N_k(x, \,.\,)$ provides a fundamental solution for $\Delta_k$ on $\mathbb R^d$ in the sense that
$\, \Delta_k(-N_k(x,.)\omega_k) = \delta_x\,$ in $\mathcal D^\prime(\RR^d).$
This implies that $-G_k(x,.)\,$ provides a fundamental solution for $\Delta_k$ in $\BB$:
$$
\Delta_k\left(-G_k(x,\cdot)\omega_k\right)=\delta_x\quad\text{in $\mathcal{D}'(\BB)$}.
$$
 Our next result provides sharp two-sided bounds for $G_k(x,y)$ which are more convenient to deal with rather than \eqref{GreenBallDiff}. For $x\in\BB$ denote
  $\delta(x):=1-|x|$.

\begin{thm}\label{thm:GreenEst}
The two-sided bound of $G_k(x,y)$ on $\BB\times\BB$ is given by
$$
G_k(x,y)\asymp\int_{C(y)}\frac{(1-|x|^2)(1-|y|^2)d\mu_y(z)}{\left(1+|x|^2|y|^2-2\langle x, z\rangle\right)\left(|x|^2+|y|^2-2\langle x, z\rangle\right)^{\gamma+d/2-1}}
$$
$$
\asymp\int_{C(y)}\frac{\delta(x)\delta(y)d\mu_y(z)}{\left(\delta(x)\delta(y)+|x|^2+|y|^2-2\langle x, z\rangle\right)\left(|x|^2+|y|^2-2\langle x, z\rangle\right)^{\gamma+d/2-1}}.
$$
\end{thm}

\begin{proof}
Note that for $x,y\in\BB$ we have 
$$
1+|x|^2|y|^2-|x|^2-|y|^2=(1-|x|^2)(1-|y|^2)\asymp \delta(x)\delta(y).
$$
Hence, the estimate is a direct consequence of Theorem \ref{thm:GreenBall} and Lemma \ref{TechLem} below. 
\end{proof}

\begin{lem}\label{TechLem}
Fix $p>0$. There exists a constant $C_p>0$ depending only on $p$ such that for all $0<a<b<\infty$ we have
$$
\frac{b-a}{C_pba^p}\leq \frac{1}{a^{p}}-\frac{1}{b^{p}}\leq \frac{C_p(b-a)}{ba^p}.
$$
\end{lem}

\begin{proof}
Assume first $p>1$. Then by \cite[Lemma 6, (11)]{BDL} (see also \eqref{elem-ineq}) we get
\begin{align*}
b^p-a^p\leq &C(b-a)^2b^{p-2}+pa^{p-1}(b-a)\\
 \leq &C(b-a)(b^{p-1}-ab^{p-2}+a^{p-1})\\
 = &Cb^{p-1}(b-a)\left(1+\left(a/b\right)^{p-1}-a/b\right),
\end{align*}
and the lower bound obtains analogously. Furthermore, since $p>1$, we have 
$$
\sup_{x\in[0,1]}|x^{p-1}-x|<1.
$$
Hence $b^p-a^p\asymp b^{p-1}(b-a)$ and
$$
\frac{1}{a^{p}}-\frac{1}{b^{p}}=\frac{b^p-a^p}{(ab)^p} \asymp \frac{b-a}{ba^p}.
$$
Here $\asymp$ means two-sided estimates with constants depending only on $p$. For $0<p\leq1$ we let $q=p+1$. We have
$$
\frac{1}{a^{p}}-\frac{1}{b^{p}}=\frac{a}{a^{q}}-\frac{b}{b^{q}}=\frac{ab^q-ba^q}{(ab)^q}.
$$
Let $c=b^{1/q}a$, $d=a^{1/q}b$. Then $0<c<d<\infty$ and applying the estimate obtained previously we get
\begin{align*}
ab^q-ba^q=d^q-c^q\asymp &d^{q-1}(d-c)\\
=&a^{(q-1)/q}b^{q-1}\left(a^{1/q}b-b^{1/q}a\right)\\
=&ab^{q-1}\left(b-b^{1/q}a^{1-1/q}\right).
\end{align*}
Since $a<b$, we obtain
$$
b-b^{1/q}a^{1-1/q}=b-\left(b/a\right)^{1/q}a\leq b-a
$$
and the upper bound follows. To get the lower bound define $f(x)=b^{1/q}x^{1-1/q}$ for $x\in[a,b]$. Then $f'(x)=(1-1/q)(b/x)^{1/q}$ and
 by the mean value theorem, for some $\xi\in(a,b)$ we have
$$
b-b^{1/q}a^{1-1/q}=f(b)-f(a)=(1-1/q)(b/\xi)^{1/q}(b-a)\geq (1-1/q)(b-a).
$$
Therefore
$$
\frac{ab^q-ba^q}{(ab)^q}\asymp \frac{ab^{q-1}\left(b-a\right)}{(ab)^q}=\frac{b-a}{ba^p}.
$$
\end{proof}

\noindent A simple consequence of Theorem \ref{thm:GreenEst} is the following estimate.

\begin{cor} 
Let $y_0\in\BB$ be fixed. There is a constant $C>0$ depending on $d$, $k$ and $y_0$ only, such that
$$
C^{-1}\delta(x)N_k(x,y_0)\leq G_k(x,y_0)\leq C\delta(x)N_k(x,y_0).
$$
\end{cor}


\noindent The following classical formula relates the Poisson kernel $P_k(x,y)$ to the Green function $G_k(x,y)$.

\begin{prop}\label{thm:GreenPoisson}
For all $x\in\BB$ and $y\in\s$ we have
\begin{equation}\nonumber
P_k(x,y)=-d_k\langle y,\nabla_y G_k(x,y)\rangle.
\end{equation}
\end{prop}

\begin{proof}
We use the symmetry $G_k(x,y)=G_k(y,x)$. By the dominated convergence, we can differentiate under the integral sign in \eqref{GreenBallDiff} to see that for all $x\in\BB$ and $y\in\s$,
$$
-d_k\langle y,\nabla_y G_k(x,y)\rangle=(1-|x|^2)\int_{C(x)}\left(|x|^2+1-2\langle y, z\rangle
\right)^{-\gamma-d/2}  d\mu_x(z).
$$
With $f(x)= |x|^{-2\gamma-d}\,$ and in view of \eqref{radialminus} and representation \eqref{PoissonRoesler} for the kernel $P_k$ we obtain
\[ -d_k\langle y,\nabla_y G_k(x,y)\rangle = (1-|x|^2) f(x*_k-y) = P_k(x,y).
 \]
\end{proof}

\section{Poisson-Jensen formula and Hardy-Stein\\ identities}

Our first goal in this section is to prove the so-called {\it Poisson-Jensen formula} for $\Delta_k$-subharmonic functions on $\overline\BB$. 
The corresponding result for classical subharmonic functions may be found in \cite{HaKe}. We will next use the formula to derive the {\it Hardy-Stein identites} for $\Delta_k$-harmonic functions on $\BB$, which equivalently characterize the Hardy spaces of $\Delta_k$ in the spirit of \cite{BDL}.

All functions in this section are assumed to be real-valued. Let $\Omega\subset\RR^d$ be a $W$-invariant open set. We will say that a function $u\in C^2(\Omega)$ is $\Delta_k$-subharmonic on 
$\Omega$ if $\Delta_k u(x)\geq0$ for all $x\in\Omega$. We refer to \cite{Re} for basic properties and other characterizations of 
$\Delta_k$-subharmonic functions. We will further say that a function $u$ is $\Delta_k$-harmonic 
(resp. $\Delta_k$-subharmonic) on $\overline\BB$ if there exists $\varepsilon>0$ such that $u$ extends to a 
$\Delta_k$-harmonic (resp. $\Delta_k$-subharmonic) function on $\BB_\varepsilon:=\{x:|x|<1+\varepsilon\}$. For $r>0$ we define the dilation of a function $u$ by $u_r(x):=u(rx)$.

The Riesz decomposition theorem \cite[Theorem 2.74, see also Ex. 2.47 and Corollary 2.53]{Re} implies that for every 
$\varepsilon>0$ and every function $u$ which is $\Delta_k$-subharmonic on $\BB_\varepsilon:=\{x:|x|<1+\varepsilon\}$ 
 there exists a unique $\Delta_k$-harmonic function $h_\varepsilon$ on $\BB_{\varepsilon/2}\subset\overline\BB_{\varepsilon/2}\subset\BB_\varepsilon$ such that
\begin{equation}\label{eq:Riesz}
u(x)=-\int_{\BB_{\varepsilon/2}} N_k(x,y)\Delta_ku(x)\omega_k(x)dx+h_\varepsilon(x),\quad x\in\BB_{\varepsilon/2}.
\end{equation}
As in the previous section, we denote by $\sigma$ the surface measure on $\s$ and let $\omega_k\sigma$ denote the measure on $\s$ given by $d\omega_k\sigma(x)=\omega_k(x)d\sigma(x)$. For $f\in L^1(\s,\omega_k\sigma)$ we define the Poisson integral of $f$ by
$$
P_k[f](x):=\frac{1}{d_k}\int_\s P_k(x,z)f(z)\omega_k(z)d\sigma(z),\quad x\in\BB.
$$
Our first result is the following property of the Newton kernel of $\Delta_k$.

\begin{lem}\label{lem:PoisNewton}
For all $x\in\BB$ we have
$$
P_k[N_k(\cdot,y)](x)=\begin{cases}
		N_k(x,y)-G_k(x,y), & \ y\in\BB,        \\
		N_k(x,y), & \ y\notin\BB.
		\end{cases}
$$
\end{lem}

\begin{proof}
For $y\in\BB$ the statement follows from \eqref{GreenDef}. Clearly, $K_k[N_k(\cdot,y)]$ is $\Delta_k$-harmonic on 
$\BB$ and continuous on $\overline\BB$ with $K_k[N_k(\cdot,y)](x)=N_k(x,y)$ for all $x\in\s$. By the uniqueness 
of the solution to the $\Delta_k$-Dirichlet problem \cite{MaYo1} we have $K_k[N_k(\cdot,y)]=P_k[N_k(\cdot,y)]$ on $\BB$. 
When $y\in(\overline\BB)^c$, then $N_k(\cdot,y)$ is $\Delta_k$-harmonic on $\overline\BB$, and hence $N_k(\cdot,y)=P_k[N_k(\cdot,y)]$ 
on $\BB$ in this case. Finally, let $y\in\s$. Since $N_k(\cdot,y)$ is $\Delta_k$-harmonic on $\BB$, the dilation $N_k(\cdot,y)_r$ 
is $\Delta_k$-harmonic on $\overline\BB$ for any $0<r<1$. Hence
$$
N_k(rx,y)=\frac{1}{d_k}\int_\s P_k(x,z)N_k(rz,y)\omega_k(z)d\sigma(z),
$$
and it is enough to show that the right-hand side above tends to $P_k[N_k(\cdot,y)](x)$ as $r\to1$. First note that Fatou's lemma gives $P_k[N_k(\cdot,y)](x)\leq N_k(x,y)$. By \eqref{DunklNewton}, for $z,y\in\s$ we have
$$
N_k(rz,y)=C_k\int_{C(y)}\left(|rz|^2+|y|^2-2\langle rz, v\rangle\right)^{1-\gamma-d/2}d\mu_y(v).
$$
For $v\in C(y)$ write $v=\sum_{g\in W}\lambda_g(v)gy$, where $\lambda_g(v)\geq0$ for all $g\in W$ and $\sum_{g\in W}\lambda_g(v)=1$. This gives
$$
|rz|^2+|y|^2-2\langle rz, v\rangle=\sum_{g\in W}\lambda_g(v)|rz-gy|^2.
$$
Furthermore, since $|z|=|gy|=1$, we have $|rz-gy|\geq |rz-rgy|$ for any $0<r<1$. Consequently,
$$
|rz|^2+|y|^2-2\langle rz, v\rangle\geq r^2(|z|^2+|y|^2-2\langle z, v\rangle),
$$
and $N_k(rz,y)\leq r^{2-2\gamma-d}N_k(z,y)$. Therefore, $N_k(rz,y)\leq CN_k(z,y)$ for all $1/2<r<1$ and 
$P_k[N_k(\cdot,y)](x)\leq N_k(x,y)<\infty$. The dominated convergence theorem gives the result.
\end{proof}

\noindent As a consequence of \eqref{eq:Riesz} and Lemma \ref{lem:PoisNewton} we obtain the following Poisson-Jensen formula.

\begin{thm}\label{thm:JensPois}
Let $u$ be $\Delta_k$-subharmonic on $\overline\BB$. Then for evey $x\in\BB$ we have
$$
u(x)=\frac{1}{d_k}\int_\s P_k(x,y)u(y)\omega_k(y)d\sigma(y)-\int_\BB G_k(x,y)\Delta_k u(y)\omega_k(y)dy.
$$
\end{thm}

\begin{proof}
Choose $\varepsilon>0$ such that $u$ extends to a $\Delta_k$-subharmonic function on $\BB_\varepsilon$. By \eqref{eq:Riesz},
$$
u(x)=-\int_{\BB_{\varepsilon/2}} N_k(x,y)\Delta_k u(x)\omega_k(x)dx+h_\varepsilon(x),\quad x\in\BB_{\varepsilon/2},
$$
where $h_\varepsilon$ is $\Delta_k$-harmonic on $\BB_{\varepsilon/2}$. Evaluating the Poisson integral of both sides and applying 
Fubini's theorem and Lemma \ref{lem:PoisNewton} we get
\begin{align*}
	P_k[u](x)=&\frac{1}{d_k}\int_\s P_k(x,y)u(y)\omega_k(y)d\sigma(y)\\
	=&\frac{1}{d_k}\int_\s P_k(x,y)\left(-\int_{\BB_{\varepsilon/2}} N_k(y,z)\Delta_k u(z)\omega_k(z)dz+h_\varepsilon(y)\right)\omega_k(y)d\sigma(y)\\
	=&-\int_{\BB_{\varepsilon/2}}\left(\frac{1}{d_k}\int_\s P_k(x,y)N_k(y,z)\omega_k(y)d\sigma(y)\right)\Delta_k u(z)\omega_k(z)dz+h_\varepsilon(x)\\
	=&\int_{\BB}\left(G_k(x,z)-N_k(x,z)\right)\Delta_k u(z)\omega_k(z)dz\\
	&-\int_{\BB_{\varepsilon/2}\setminus\BB}N_k(x,z)\Delta_k u(z)\omega_k(z)dz+h_\varepsilon(x)\\
	=&\int_{\BB}G_k(x,z)\Delta_k u(z)\omega_k(z)dz+u(x).
\end{align*}
\end{proof}
Let $1\leq p\leq\infty$. The Hardy space $H^p_k(\BB)$ is defined as the family of those $\Delta_k$-harmonic functions on $\BB$ which satisfy
$$
\|u\|_{H^p}:=\sup_{0\leq r<1}\|u_r\|_{L^p(\omega_k\sigma)}<\infty.
$$
By \cite[Theorem 2.2 and Theorem 2.3]{MaYo1}, $u\in H^p_k(\BB)$ for a given $1<p\leq\infty$ if and only if 
$u=P_k[f]$ for some $f\in L^p(\s,\omega_k\sigma)$, and in this case $\|u\|_{H^p}=\|f\|_{L^p(\omega_k\sigma)}$. This implies that
\begin{equation}\label{eq:Hpnorm}
\|u\|_{H^p}=\lim_{r\to 1}\|u_r\|_{L^p(\omega_k\sigma)}
\end{equation}
for any $\Delta_k$-harmonic function $u$ on $\BB$. As an application of Theorem \ref{thm:JensPois}, we will give an equivalent characterization of the spaces $H^p_k(\BB)$, $1<p<\infty$, in terms of the Hardy-Stein identities. The approach is inspired by \cite{BDL}, where similar description was obtained for Hardy spaces of the classical Laplacian $\Delta$ and the fractional Laplacian $\Delta^{\alpha/2}$.

Let $1<p<\infty$. For $a, b\in \RR$ we set
\begin{equation}\label{eq:defF}
 F(a,b) = |b|^p-|a|^p - pa|a|^{p-2}(b-a)\,.
\end{equation}
Here $F(a,b)=|b|^p$ if $a=0$, and $F(a,b)=(p-1)|a|^p$ if $b=0$.
For instance, if $p=2$, then $F(a,b)=(b-a)^2$.
Generally, $F(a,b)$ is the second-order Taylor remainder of $\RR\ni x\mapsto |x|^p$, therefore by convexity, $F(a,b)\geq 0$.
Furthermore, for $1<p<\infty$ and 
$\varepsilon\in \RR$ 
we define
\begin{equation}\label{eq:Feps}
F_\varepsilon(a,b) = 
(b^2+\varepsilon^2)^{p/2}-(a^2+\varepsilon^2)^{p/2}
  - pa(a^2+\varepsilon^2)^{(p-2)/2}(b-a)\,.
\end{equation}
Since $F_\varepsilon (a,b)$ is the second-order Taylor remainder of $\RR\ni x\mapsto (x^2+\varepsilon^2)^{p/2}$, by convexity, $F_\varepsilon(a,b)\geq 0$.
Of course, $F_\varepsilon(a,b)\to F_0(a,b)=F(a,b)$ as $\varepsilon\to 0$. The next result is proved in \cite[Lemma 6]{BDL}.
\begin{lem}\label{lem:F}
For every $p>1$ there is a constant $C>0$ depending on $p$ only such that
\begin{equation}\label{elem-ineq}
C^{-1}(b-a)^2(|b|\vee |a|)^{p-2}\le F(a,b) \le
C(b-a)^2(|b|\vee |a|)^{p-2},\quad
a, b\in\RR.
\end{equation} 
If $p\in (1,2)$, then 
\begin{equation}\label{eq:ub}
0\le F_\varepsilon(a,b)\leq \frac{1}{p-1}F(a,b),\quad \varepsilon, a, b\in\RR.
\end{equation}
\end{lem}

\noindent The following explicit formulas shed some light on the meaning of the function $F$.

\begin{lem}\label{lem:PthPower}
Let $u$ be of class $C^2$ in the neighborhood of $x\in\RR^d$. Then for $2\leq p<\infty$ we have
\begin{align}\nonumber
\Delta_k|u(x)|^p=p(p-1)|u(x)|^{p-2}|\nabla u(x)|^2+&2\sum_{\alpha\in R_+}k(\alpha)\frac{F(u(x),u(\sigma_\alpha(x)))}{\langle\alpha,x\rangle^2}\\
+&pu(x)|u(x)|^{p-2}\Delta_ku(x).\label{eq:PthPower1}
\end{align}
When $1<p<\infty$ and $\varepsilon>0$, then
\begin{align}\label{eq:PthPower2}
\Delta_k|u(x)+i\varepsilon|^p&=p|u(x)+i\varepsilon|^{p-4}\left[(p-1)u(x)^2+\varepsilon^2\right]|\nabla u(x)|^2\\
+&2\sum_{\alpha\in R_+}k(\alpha)\frac{F_{\varepsilon}(u(x),u(\sigma_\alpha(x)))}{\langle\alpha,x\rangle^2}
+pu(x)|u(x)+i\varepsilon|^{p-2}\Delta_ku(x).\nonumber
\end{align}
\end{lem}

\begin{proof}
When $2\leq p<\infty$ or $u(x)\neq0$ we write $|u(x)|^p=(u(x)^2)^{p/2}$ and a straightforward calculation gives
\begin{align*}
\nabla|u(x)|^p&=pu(x)|u(x)|^{p-2}\nabla u(x),\\
\Delta|u(x)|^p&=p(p-1)|u(x)|^{p-2}|\nabla u(x)|^2+pu(x)|u(x)|^{p-2}\Delta u(x).
\end{align*}
Note that
$$
|u(\sigma_\alpha(x))|^p-|u(x)|^p=F(u(x),u(\sigma_\alpha(x)))+pu(x)|u(x)|^{p-2}(u(\sigma_\alpha(x))-u(x)).
$$
Hence
\begin{align*}
\Delta_k|u(x)|^p&=\Delta|u(x)|^p+2\sum_{\alpha\in R_+}k(\alpha)\left(\frac{\langle\nabla|u(x)|^p,\alpha\rangle}{\langle\alpha,x\rangle}+\frac{|u(\sigma_\alpha(x))|^p-|u(x)|^p}{\langle\alpha,x\rangle^2}\right)\\
&=p(p-1)|u(x)|^{p-2}|\nabla u(x)|^2+pu(x)|u(x)|^{p-2}\Delta u(x)\\
&+2pu(x)|u(x)|^{p-2}\sum_{\alpha\in R_+}k(\alpha)\left(\frac{\langle\nabla u(x),\alpha\rangle}{\langle\alpha,x\rangle}+\frac{(u(\sigma_\alpha(x))-u(x))}{\langle\alpha,x\rangle^2}\right)\\
&+2\sum_{\alpha\in R_+}k(\alpha)\frac{F(u(x),u(\sigma_\alpha(x)))}{\langle\alpha,x\rangle^2},
\end{align*}
and \eqref{eq:PthPower1} follows. For $1<p<\infty$ and $\varepsilon>0$ we have 
\begin{align*}
\nabla|u(x)+i\varepsilon|^{p}&=pu(x)|u(x)+i\varepsilon|^{p-2}\nabla u(x),\\
\Delta|u(x)+i\varepsilon|^{p}&=p|u(x)+i\varepsilon|^{p-4}\left[(p-1)u(x)^2+\varepsilon^2\right]|\nabla u(x)|^2\\
&+p|u(x)+i\varepsilon|^{p-2}u(x)\Delta u(x),
\end{align*}
and
\begin{align*}
|u(\sigma_\alpha(x))+i\varepsilon|^p-|u(x)+i\varepsilon|^p&=F_\varepsilon(u(x),u(\sigma_\alpha(x)))\\
&+pu(x)|u(x)+i\varepsilon|^{p-2}(u(\sigma_\alpha(x))-u(x)).
\end{align*}
The rest of the proof is similar to the previous case.
\end{proof}

\noindent We are now ready to prove the Hardy-Stein identities.

\begin{thm}\label{thm:HardyStein}
Let $1<p<\infty$. Then for any $u\in H^p_k(\BB)$ we have
\begin{align*}
\|u\|^p_{H^p}=|u(0)|^p+C_k\int_\BB (|y|^{2-2\gamma-d}-1)&[p(p-1)|u(y)|^{p-2}|\nabla u(y)|^2\\
+2\sum_{\alpha\in R_+}&k(\alpha)\frac{F(u(y),u(\sigma_\alpha(y)))}{\langle\alpha,y\rangle^2}]\omega_k(y)dy.
\end{align*}
In fact, a $\Delta_k$-harmonic function $u$ on $\BB$ belongs to $H^p_k(\BB)$ if and only if the integral above is finite.
\end{thm}

\begin{proof}
Suppose $v$ is $\Delta_k$-subharmonic on $\BB$. Then $v_r$ is $\Delta_k$-subharmonic on $\overline\BB$ for any $0<r<1$. By Theorem~\ref{thm:JensPois},
\begin{equation}\label{eq:HSdilation}
v(0)=\frac{1}{d_k}\int_\s v_r(y)\omega_k(y)d\sigma(y)-\int_\BB G_k(0,y)(\Delta_k v_r)(y)\omega_k(y)dy.
\end{equation}
Since $(\Delta_k v_r)(x)=r^2(\Delta_k v)_r(x)$, by \eqref{GreenBallDiff} we have
\begin{align}\nonumber
&\int_\BB G_k(0,y)(\Delta_k v_r)(y)\omega_k(y)dy=C_kr^2\int_\BB (|y|^{2-2\gamma-d}-1)(\Delta_k v)(ry)\omega_k(y)dy\\
&=C_k\int_{B(0,r)} (|z|^{2-2\gamma-d}-r^{2-2\gamma-d})\Delta_k v(z)\omega_k(z)dz,\label{eq:GreenDilation}
\end{align}
where $B(0,r):=\{x\in\RR^d:|x|<r\}$. Let now $u$ be $\Delta_k$-harmonic on $\BB$ and suppose first $2\leq p<\infty$. Then $|u|^p$ is of class $C^2$ on $\BB$ and by \eqref{eq:PthPower1} we have 
$$
\Delta_k|u(x)|^p=p(p-1)|u(x)|^{p-2}|\nabla u(x)|^2+2\sum_{\alpha\in R_+}k(\alpha)\frac{F(u(x),u(\sigma_\alpha(x)))}{\langle\alpha,x\rangle^2}.
$$
In particular, $\Delta_k|u|^p\geq0$ on $\BB$ so \eqref{eq:HSdilation} and \eqref{eq:GreenDilation} apply to $v=|u|^p$. Let $r\to1$. By \eqref{eq:Hpnorm},
$$
\frac{1}{d_k}\int_\s |u(ry)|^p\omega_k(y)d\sigma(y)\to\|u\|_{H^p}^p,
$$
and by the monotone convergence,
\begin{align*}
C_k\int_{B(0,r)} (|y|^{2-2\gamma-d}-r^{2-2\gamma-d})&\Delta_k|u(z)|^p\omega_k(z)dz\\
&\to\int_{\BB}G_k(0,z)\Delta_k|u(z)|^p\omega_k(z)dz.
\end{align*}
This gives the result for $p\geq2$. Assume now $1<p<2$ and let $\varepsilon>0$. Then $|u+i\varepsilon|^p$ is of class $C^2$ on $\BB$ and by \eqref{eq:PthPower2} we have
\begin{align*}
\Delta_k|u(x)+i\varepsilon|^p=p|u(x)+i\varepsilon|^{p-4}&\left[(p-1)u(x)^2+\varepsilon^2\right]|\nabla u(x)|^2\\
+&2\sum_{\alpha\in R_+}k(\alpha)\frac{F_{\varepsilon}(u(x),u(\sigma_\alpha(x)))}{\langle\alpha,x\rangle^2}.
\end{align*}
Since $\Delta_k|u+i\varepsilon|^p\geq0$ on $\BB$, we can apply \eqref{eq:HSdilation} and \eqref{eq:GreenDilation} to $v=|u+i\varepsilon|^p$. This gives
\begin{align*}
|u(0)+i\varepsilon|^p=&\frac{1}{d_k}\int_\s |u(ry)+i\varepsilon|^p\omega_k(y)d\sigma(y)\\
-&C_k\int_{B(0,r)} (|y|^{2-2\gamma-d}-r^{2-2\gamma-d})\Delta_k|u(y)+i\varepsilon|^p\omega_k(y)dy.
\end{align*}
Let $\varepsilon\to0$. Then 
$$
\Delta_k|u(x)+i\varepsilon|^p\to p(p-1)|u(x)|^{p-2}|\nabla u(x)|^2+2\sum_{\alpha\in R_+}k(\alpha)\frac{F(u(x),u(\sigma_\alpha(x)))}{\langle\alpha,x\rangle^2}
$$
for a.e. $x\in\BB$, and 
$$
\int_\s |u(ry)+i\varepsilon|^p\omega_k(y)d\sigma(y)\to\int_\s |u(ry)|^p\omega_k(y)d\sigma(y).
$$
Fatou's lemma, \eqref{eq:ub} and dominated convergence give 
\begin{align*}
|u(0)|^p&=\frac{1}{d_k}\int_\s |u(ry)|^p\omega_k(y)d\sigma(y)-C_k\int_{B(0,r)} (|y|^{2-2\gamma-d}-r^{2-2\gamma-d})\\
\times&[p(p-1)|u(x)|^{p-2}|\nabla u(x)|^2+2\sum_{\alpha\in R_+}k(\alpha)\frac{F(u(x),u(\sigma_\alpha(x)))}{\langle\alpha,x\rangle^2}]\omega_k(y)dy.
\end{align*}
Let $r\to1$. The final conclusion follows from \eqref{eq:Hpnorm} and monotone convergence.
\end{proof}

\noindent An immediate consequence of Theorem \ref{thm:HardyStein} and \cite[Theorem 2.2 and Theorem 2.3]{MaYo1} is the following identity.

\begin{cor}
Let $1<p<\infty$, $f\in L^p(\s,\omega_k\sigma)$ and set $u=P_k[f]$. Then
\begin{align*}
\int_\s|f(x)|^p\omega_k(x)d\sigma(x)=&|u(0)|^p+C_k\int_\BB (|y|^{2-2\gamma-d}-1)\\
\times[p(p-1)|u(y)|^{p-2}&|\nabla u(y)|^2
+2\sum_{\alpha\in R_+}k(\alpha)\frac{F(u(y),u(\sigma_\alpha(y)))}{\langle\alpha,y\rangle^2}]\omega_k(y)dy.
\end{align*}
\end{cor}

\section{Sharp estimates of the Green function and\\ Poisson kernel in rank one}

In this part we consider the rank one case. The basic situation is that of the root system 
$A_1 = \{\pm(e_1-e_2)\}$
in $\mathbb R^2$, where $e_1, e_2$ denote the standard basis vectors. We choose $\alpha = e_1 - e_2$ as
positive root and let $\sigma_\alpha(x_1,x_2)=(x_2,x_1)$ denote the reflection corresponding to $\alpha$. To simplify formulas, it will be convenient to switch to the orthonormal basis
$$
(e_1',e_2')=(\frac{1}{\sqrt{2}}(e_1-e_2), \frac{1}{\sqrt{2}}(e_1+e_2))
$$
and write $x\in \mathbb R^2$ as $x = (x_1, x_2)$ with coordinates $x_1,
x_2$ with respect to the
basis $(e_1',e_2')$. The reflection $\sigma$ writes
$\sigma(x_1, x_2) = (-x_1, x_2)$. By formula \eqref{V_dimension_one} we obtain
$$
V_kf(y)=c_k\int_{-1}^1 f(ty_1,y_2)(1-t)^{k-1}(1+t)^k dt.
$$
This case has a nice motivation, namely the potential theory of a $2$-dimensional $k$-Dyson
Brownian Motion, which corresponds to the $W$-invariant Dunkl process in this case.

More generally, we will consider the rank one case with root system $A_1$ in $\RR^d$, with the intertwining operator given by
\begin{equation}\label{IntertwinerRank1}
V_kf(y)=c_k\int_{-1}^1 f(ty_1,y_2,...,y_d)(1-t)^{k-1}(1+t)^k dt.
\end{equation}
Though this generalization seems elementary from the algebraic point of view, it
reveals nontrival analytic phenomena which are strongly dependent on the underlying dimension. Note that in the rank one case we have $\gamma=k$, and as before we work under the assumption $d+2k>2$. 

The Newton kernel \eqref{DunklNewton} can be written as
\begin{equation}\label{DunklNewtonR1}
N_k(x,y)=\widetilde C_k\int_{-1}^1\frac{(1-t)^{k-1}(1+t)^kdt}{\left(|x|^2+|y|^2-2(tx_1y_1+x_2y_2+...+x_dy_d)\right)^{k+d/2-1}},
\end{equation}
where $\widetilde C_k=c_kC_k$ and the constants $c_k$, $C_k$ were defined in \eqref{V_dimension_one} and \eqref{constant_C_k}. The reflection $\sigma$ writes
$$
\sigma(x_1,x_2,...,x_d)=(-x_1,x_2,...,x_d).
$$
We then have
\begin{align}\label{NormId}
|x|^2+|y|^2-2(tx_1y_1+x_2y_2+...+x_dy_d)=&|x-y|^2+2x_1y_1(1-t)\\
=&|x-\sigma y|^2-2x_1y_1(1+t).\nonumber
\end{align}
Our first result in this section characterizes the asymptotic behaviour of the Newton kernel $N_k(x,y)$.

\begin{thm}\label{thm:NewtonEst}
Let $\Phi(x,y):=|x-y|\vee|x-\sigma y|$. The two-sided bound of $N_k(x,y)$ on $\RR^d\times\RR^d$ is the following.
\begin{enumerate}
\item[1.] If $d=2$, then
\begin{equation}\label{est:NewtonD2}
N_k(x,y)\asymp \frac{1}{\Phi(x,y)^{2k}}\left[1\vee\log\left(\frac{|x_1y_1|}{|x- y|^2}\right)\right].
\end{equation}
\item[2.] If $d=3$, then
\begin{equation}\label{est:NewtonD3}
N_k(x,y)\asymp \frac{1}{\Phi(x,y)^{2k}|x-y|}.
\end{equation}
\item[3.] If $d=4$, then
\begin{equation}\label{est:NewtonD4}
N_k(x,y)\asymp \frac{1}{\Phi(x,y)^{2k}|x-y|^2}\left[1\vee\log\left(\frac{|x_1y_1|}{|x-\sigma y|^2}\right)\right].
\end{equation}
\item[4.] If $d\geq5$, then
\begin{equation}\label{est:NewtonD5}
N_k(x,y)\asymp \frac{1}{\Phi(x,y)^{2k}|x-y|^2(|x-y|\wedge|x-\sigma y|)^{d-4}}.
\end{equation}
\end{enumerate}
\end{thm}


\noindent Theorem \ref{thm:NewtonEst} is a direct consequence of Lemma \ref{lem:DunklNewtonEst1} and 
Lemma \ref{lem:DunklNewtonEst2} below.


\begin{lem}\label{lem:DunklNewtonEst1}
The two-sided bound of $N_k(x,y)$ on $\left\{(x,y)\in\RR^d\times\RR^d:x_1y_1\geq0\right\}$ is as follows.
\begin{enumerate}
\item[1.] If $d=2$, then
\begin{equation}\label{DunklNewtonD2right}
N_k(x,y)\asymp\frac{1}{|x-\sigma y|^{2k}}\left[1\vee\log\left(\frac{x_1y_1}{|x-y|^2}\right)\right].
\end{equation}
\item[2.] If $d\geq3$, then
\begin{equation}\label{DunklNewtonD3right}
N_k(x,y)\asymp \frac{1}{|x-\sigma y|^{2k}|x-y|^{d-2}}.
\end{equation}
\end{enumerate}
\end{lem}

\begin{proof}
Denote $\zeta=|x-y|^{2}$ and $\eta=x_1y_1$. Since $x_1y_1\geq0$ we have $\zeta+\eta\asymp |x-\sigma y|^{2}$. By \eqref{DunklNewtonR1} and \eqref{NormId} we have
\begin{equation}\label{NewtonPotentialRank1right}
N_k(x,y)=C_k\int_{-1}^1 \frac{(1-t)^{k-1}(1+t)^kdt}{(\zeta+2\eta(1-t))^{k+d/2-1}}
=C_k\int_{0}^2 \frac{s^{k-1}(2-s)^k ds}{(\zeta+2\eta s)^{k+d/2-1}}.
\end{equation}
We write $N_k(x,y)=C_k(I_1+I_2)$, where
$$
I_1=\int_{0}^1 \frac{s^{k-1}(2-s)^k
ds}{(\zeta+2\eta s)^{k+d/2-1}}\asymp\int_{0}^1
\frac{s^{k-1}ds}{(\zeta+\eta s)^{k+d/2-1}},
$$
and
\begin{align*}
I_2=\int_{1}^2 \frac{s^{k-1}(2-s)^k ds}{(\zeta+2\eta s)^{k+d/2-1}}
\asymp \int_{1}^2 \frac{(2-s)^k ds}{(\zeta+\eta s)^{k+d/2-1}}\asymp(\zeta+\eta)^{1-k-d/2}.
\end{align*}
For $\eta=0$ the estimates of the lemma are obvious, so assume $\eta>0$. Using the change of variables $s=u\zeta$ we get
\begin{align*}
I_1\asymp &\int_{0}^{1/\zeta}\frac{\zeta^k u^{k-1} du}{(\zeta+\zeta\eta u)^{k+d/2-1}}=\zeta^{1-d/2}\int_{0}^{1/\zeta}\frac{u^{k-1}du}{(1+\eta u)^{k+d/2-1}}\\
=& \zeta^{1-d/2}\int_{0}^{1/\zeta}\frac{du}{u^{d/2}(1/u+\eta)^{k+d/2-1}}
= \zeta^{1-d/2}\int_{\zeta}^{\infty}\frac{w^{d/2-2}dw}{(w+\eta)^{k+d/2-1}}.
\end{align*}
Let $d=2$ and assume first $\eta\leq\zeta$. Then
$$
I_1\asymp\int_{\zeta}^{\infty}\frac{dw}{w(w\vee \eta)^{k}}=\int_{\zeta}^{\infty}w^{-k-1}dw
=\zeta^{-k}/k\asymp (\zeta+\eta)^{-k},
$$
and note that the same two-sided estimate holds also for $I_2$. Assume $\eta>\zeta$. We have
\begin{align*}
I_1\asymp \int_{\zeta}^{\infty}\frac{dw}{w(w\vee \eta)^{k}}=
\int_{\zeta}^{\eta}\frac{dw}{w\eta^{k}}+\int_{\eta}^{\infty}\frac{dw}{w^{k+1}}
=& \eta^{-k}\log\left(\eta/\zeta\right)+\eta^{-k}/k\\
\asymp &(\zeta+\eta)^{-k}\left[1\vee\log\left(\eta/\zeta\right)\right].
\end{align*}
It is clear that the estimate above holds also for $I_1+I_2$, and combining it with the previous case we get \eqref{DunklNewtonD2right}.

Assume $d\geq3$. For $\eta\leq\zeta$ we get
\begin{align*}
I_1\asymp \zeta^{1-d/2}\int_{\zeta}^{\infty}\frac{w^{d/2-2}dw}{(w\vee\eta)^{k+d/2-1}}
=&\zeta^{1-d/2}\int_{\zeta}^{\infty}w^{-k-1}dw\\
=& \zeta^{1-k-d/2}/k\asymp(\zeta+\eta)^{-k}\zeta^{1-d/2}.
\end{align*}
When $\eta>\zeta$, then a similar reasoning as before gives
\begin{align*}
I_1\asymp &\zeta^{1-d/2}\left(\int_{\zeta}^{\eta}\frac{w^{d/2-2}dw}{\eta^{k+d/2-1}}+\int_{\eta}^{\infty}\frac{dw}{w^{k+1}}\right)\\
=&\zeta^{1-d/2}\left[\frac{2}{(d-2)\eta^{k+d/2-1}}\left(\eta^{d/2-1}-\zeta^{d/2-1}\right)+\frac{1}{k\eta^{k}}\right]\\
=&\frac{1}{\eta^{k+d/2-1}\zeta^{d/2-1}}\left[\frac{2}{d-2}\left(\eta^{d/2-1}-\zeta^{d/2-1}\right)+\frac{\eta^{d/2-1}}{k}\right].
\end{align*}
Since $0<\eta^{d/2-1}-\zeta^{d/2-1}\leq \eta^{d/2-1}$, we obtain
$$
I_1\asymp\eta^{-k}\zeta^{1-d/2}\asymp(\zeta+\eta)^{-k}\zeta^{1-d/2}.
$$
Finally,
$$
I_2\asymp (\zeta+\eta)^{1-k-d/2}\leq(\zeta+\eta)^{-k}\zeta^{1-d/2},
$$
and hence $I_1+I_2\asymp I_1$. This proves \eqref{DunklNewtonD3right}.
\end{proof}

\vspace{5mm}

\begin{lem}\label{lem:DunklNewtonEst2}
The two-sided bound of $N_k(x,y)$ on $\left\{(x,y)\in\RR^d\times\RR^d:x_1y_1<0\right\}$ is as follows.
\begin{enumerate}
\item[1.] If $2\leq d\leq3$, then
\begin{equation}\label{DunklNewtonD23left}
N_k(x,y)\asymp \frac{1}{|x-y|^{2k+d-2}}.
\end{equation}
\item[2.] If $d=4$, then
\begin{equation}\label{DunklNewtonD4left}
N_k(x,y)\asymp \frac{1}{|x-y|^{2k+2}}\left[1\vee\log\left(\frac{|x_1y_1|}{|x-\sigma y|^2}\right)\right].
\end{equation}
\item[3.] If $d\geq5$, then 
\begin{equation}\label{DunklNewtonD5left}
N_k(x,y)\asymp \frac{1}{|x-y|^{2k+2}|x-\sigma y|^{d-4}}.
\end{equation}
\end{enumerate}
\end{lem}

\begin{proof}
Denote $\zeta=|x-\sigma y|^{2}$ and $\eta=|x_1y_1|$. Since $x_1y_1<0$ we have $\zeta+\eta\asymp |x-y|^{2}$. By \eqref{DunklNewtonR1} and \eqref{NormId} we have
\begin{align}\label{NewtonPotentialRank1Left}
N_k(x,y)=&C_k\int_{-1}^1 \frac{(1-t)^{k-1}(1+t)^kdt}{(\zeta+2\eta(1+t))^{k+d/2-1}}=C_k\int_{0}^2 \frac{(2-s)^{k-1}s^k ds}{(\zeta+2\eta s)^{k+d/2-1}}.
\end{align}
We write $N_k(x,y)=C_k(I_1+I_2)$, where
$$
I_1=\int_{0}^1 \frac{(2-s)^{k-1}s^k
ds}{(\zeta+2\eta s)^{k+d/2-1}}\asymp\int_{0}^1
\frac{s^{k}ds}{(\zeta+\eta s)^{k+d/2-1}},
$$
and
$$
I_2=\int_{1}^2 \frac{(2-s)^{k-1}s^kds}{(\zeta+2\eta s)^{k+d/2-1}}
\asymp \int_{1}^2 \frac{(2-s)^{k-1} ds}{(\zeta+\eta s)^{k+d/2-1}}
\asymp (\zeta+\eta)^{1-k-d/2}.
$$
As in the proof of Lemma \ref{lem:DunklNewtonEst1}, we apply the change of variables $s=u\zeta$ and get
\begin{align*}
I_1\asymp\zeta^{2-d/2}\int_{0}^{1/\zeta}\frac{u^{k}du}{(1+\eta u)^{k+d/2-1}}
=& \zeta^{2-d/2}\int_{0}^{1/\zeta}\frac{du}{u^{d/2-1}(1/u+\eta)^{k+d/2-1}}\\
=& \zeta^{2-d/2}\int_{\zeta}^{\infty}\frac{w^{d/2-3}dw}{(w+\eta)^{k+d/2-1}}.
\end{align*}
Assume $\eta\leq\zeta$. Then
\begin{equation}\label{EstI1}
I_1\asymp \zeta^{2-d/2}\int_{\zeta}^{\infty}w^{-k-2}dw
\asymp \zeta^{1-k-d/2}
\asymp (\zeta+\eta)^{1-k-d/2}.
\end{equation}
When $\eta>\zeta$ we have
\begin{align*}
I_1\asymp &\zeta^{2-d/2}\int_{\zeta}^{\infty}\frac{w^{d/2-3}dw}{(w\vee\eta)^{k+d/2-1}}
=\zeta^{2-d/2}\left(\int_{\zeta}^{\eta}\frac{w^{d/2-3}dw}{\eta^{k+d/2-1}}+\int_{\eta}^{\infty}\frac{dw}{w^{k+2}} \right)\\
=&\zeta^{2-d/2}\left(\frac{1}{\eta^{k+d/2-1}}\int_{\zeta}^{\eta}w^{d/2-3}dw+\frac{1}{(k+1)\eta^{k+1}}\right).
\end{align*}
Assume first $2\leq d\leq3$. We obtain
\begin{align*}
I_1\asymp &\zeta^{2-d/2}\left[\frac{2}{(4-d)\eta^{k+d/2-1}}\left(\frac{1}{\zeta^{2-d/2}}-\frac{1}{\eta^{2-d/2}}\right)
+\frac{1}{(k+1)\eta^{k+1}}\right]\\
=&\frac{2}{(4-d)\eta^{k+d/2-1}}\left[1-\left(1-
\frac{(4-d)}{2(k+1)}\right)\frac{\zeta^{2-d/2}}{\eta^{2-d/2}}\right].
\end{align*}
Since $0<(4-d)/(2k+2)<1$ and $\zeta^{2-d/2}\leq \eta^{2-d/2}$, we get
$$
I_1\asymp \eta^{1-k-d/2}\asymp (\zeta+\eta)^{1-k-d/2}.
$$
Note that the same estimate holds for $I_1$ when $\eta\leq\zeta$ and for $I_2$ for all $x,y$ with $x_1y_1<0$. This gives \eqref{DunklNewtonD23left}. 

For $d=4$ and $\eta>\zeta$ we have
$$
I_1\asymp \eta^{-k-1}\log\left(\eta/\zeta\right)
+\eta^{-k-1}/(k+1)
\asymp (\zeta+\eta)^{-k-1}\left[1\vee\log\left(\eta/\zeta\right)\right].
$$
The last two-sided estimate remains valid for $I_1$ also when $\eta\leq\zeta$ and the upper bound dominates $I_2$. This proves \eqref{DunklNewtonD4left}. 

Finally, assume $d\geq5$ and $\eta>\zeta$. Then
\begin{align*}
I_1\asymp &\zeta^{2-d/2}\left[\frac{2\left(\eta^{d/2-2}-\zeta^{d/2-2}\right)}{(d-4)\eta^{k+d/2-1}}+\frac{1}{(k+1)\eta^{k+1}}\right]\\
\asymp &\eta^{-k-1}\zeta^{2-d/2}\left(\frac{\eta^{d/2-2}-\zeta^{d/2-2}}{\eta^{d/2-2}}+1\right)
\asymp (\zeta+\eta)^{-k-1}\zeta^{2-d/2}.
\end{align*}
For $\eta\leq\zeta$ we have $\zeta+\eta\asymp\zeta$, and by \eqref{EstI1} we get
\begin{align*}
I_1\asymp (\zeta+\eta)^{1-k-d/2}\asymp (\zeta+\eta)^{-k-1}\zeta^{2-d/2}.
\end{align*}
Since $d\geq5$, the upper bound of the last estimate also dominates $I_2$. The proof of \eqref{DunklNewtonD5left} is complete.
\end{proof}

\vspace{5mm}

\noindent We will next give sharp two-sided estimates of $G_k(x,y)$ in the rank one case.

\begin{thm}\label{thm:GreenEstR1}
Let $\Phi(x,y):=|x-y|\vee|x-\sigma y|$. The two-sided bound of $G_k(x,y)$ on $\BB\times\BB$ is the following.
\begin{enumerate}
\item[1.] If $d=2$, then
\begin{align}\label{est:GreenD2}
G_k(x,y)\asymp\frac{1}{\Phi(x,y)^{2k}}&\left(1\wedge\frac{\delta(x)\delta(y)}{|x-y|^{2}}\right)\left[1\vee\log\left(\frac{|x_1y_1|\wedge\delta(x)\delta(y)}{|x-y|^2}\right)\right]\nonumber\\
&\times
\left[1\vee\log\left(\frac{|x_1y_1|}{\delta(x)\delta(y)\vee|x-\sigma y|^2}\right)\right].
\end{align}
\item[2.] If $d=3$, then
\begin{align}\label{est:GreenD3}
G_k(x,y)\asymp\frac{1}{\Phi(x,y)^{2k}|x-y|}&\left(1\wedge\frac{\sqrt{\delta(x)\delta(y)}}{|x-y|}\right)\\
&\times\left(1\wedge\frac{\sqrt{\delta(x)\delta(y)}}{|x-y|\wedge|x-\sigma y|}\right).\nonumber
\end{align}
\item[3.] If $d=4$, then
\begin{align}\label{est:GreenD4}
G_k(x,y)\asymp\frac{1}{\Phi(x,y)^{2k}|x-y|^{2}}&\left(1\wedge\frac{\delta(x)\delta(y)}{|x-y|^{2}\wedge|x-\sigma y|^2}\right)\\
\times &
\left[1\vee\log\left(\frac{|x_1y_1|\wedge\delta(x)\delta(y)}{|x-\sigma y|^2}\right)\right].\nonumber
\end{align}
\item[4.] If $d\geq5$, then
\begin{align}\label{est:GreenD5}
G_k(x,y)&\asymp\frac{1}{\Phi(x,y)^{2k}(|x-y|\wedge|x-\sigma y|)^{d-4}|x-y|^{2}}\\
&\times\left(1\wedge\frac{\delta(x)\delta(y)}{|x-y|^{2}\wedge|x-\sigma y|^2}\right).\nonumber
\end{align}
\end{enumerate}
\end{thm}

\noindent Theorem \ref{thm:GreenEstR1} is a direct consequence of Lemma \ref{lem:GreenEst1right} and Lemma 
\ref{lem:GreenEst1left} below.

\begin{lem}\label{lem:GreenEst1right}
The two-sided bound of $G_k(x,y)$ on $\left\{(x,y)\in\BB\times\BB:x_1y_1\geq0\right\}$ is the following.
\begin{enumerate}
\item[1.] If $d=2$, then
\begin{align}\label{GreenD2right}
G_k(x,y)\asymp\frac{1}{|x-\sigma y|^{2k}}&\left(1\wedge\frac{\delta(x)\delta(y)}{|x-y|^{2}}\right)\\
&\times\left[1\vee\log\left(\frac{x_1y_1\wedge\delta(x)\delta(y)}{|x-y|^2}\right)\right].\nonumber
\end{align}
\item[2.] If $d\geq3$, then
\begin{equation}\label{GreenD3right}
G_k(x,y)\asymp \frac{1}{|x-\sigma y|^{2k}|x-y|^{d-2}}\left(1\wedge\frac{\delta(x)\delta(y)}{|x-y|^{2}}\right).
\end{equation}
\end{enumerate}
\end{lem}

\vspace{5mm}

\begin{proof}
Let $\zeta=|x-y|^{2}$, $\eta=x_1y_1$, and $\xi=\delta(x)\delta(y)$.
By Theorem \ref{thm:GreenEst}, \eqref{IntertwinerRank1}, and \eqref{NormId} we have
\begin{equation}\label{Rank1GreenAsymp}
G_k(x,y)\asymp
\int_{-1}^1 \frac{\xi(1-t)^{k-1}(1+t)^kdt}{(\xi+\zeta+\eta(1-t))(\zeta+\eta(1-t))^{k+d/2-1}}.
\end{equation}
Assume first $\xi\leq\zeta$. Then by \eqref{Rank1GreenAsymp},
$$
G_k(x,y)\asymp\int_{-1}^1 \frac{\xi(1-t)^{k-1}(1+t)^kdt}{(\zeta+\eta(1-t))^{k+d/2}},
$$
and observe that the same integral appears in \eqref{NewtonPotentialRank1right} with $d'=d+2$ instead of $d$. Hence, by \eqref{DunklNewtonD3right} we get
\begin{equation}\label{GreenEasyCase}
G_k(x,y)\asymp\frac{\xi}{(\zeta+\eta)^{k}\zeta^{d/2}}.
\end{equation}
Assume $\xi>\zeta$. Using \eqref{Rank1GreenAsymp} and the substitution $s=1-t$ we obtain
$$
G_k(x,y)\asymp
\int_{0}^2 \frac{\xi s^{k-1}(2-s)^kds}{(\xi+\eta s)(\zeta+\eta s)^{k+d/2-1}}=I_1+I_2,
$$
where
\begin{equation}\label{I1estimate}
I_1\asymp\int_{0}^1 \frac{\xi s^{k-1}ds}{(\xi+\eta s)(\zeta+\eta s)^{k+d/2-1}},
\end{equation}
and
\begin{equation}\label{I2estimate}
I_2\asymp \int_{1}^2 \frac{\xi(2-s)^{k}ds}{(\xi+\eta s)(\zeta+\eta s)^{k+d/2-1}}
\asymp \frac{\xi}{(\xi+\eta)(\zeta+\eta)^{k+d/2-1}}.
\end{equation}
In order to estimate $I_1$ we consider two cases. Assume first $d=2$.
\begin{enumerate}
\item[(a)] $\xi\geq\eta$. Then, by \eqref{I1estimate} and the estimates from the proof of Lemma \ref{lem:DunklNewtonEst1} we get
$$
I_1\asymp\int_{0}^1 \frac{s^{k-1}ds}{(\zeta+\eta s)^{k}}\asymp
\frac{1}{(\zeta+\eta)^k}\left(1\vee\log\frac{\eta}{\zeta}\right).
$$
In view of \eqref{I2estimate}, we also have $I_1+I_2\asymp I_1$.
\item[(b)] $\xi<\eta$. Then $\zeta<\xi<\eta$, and by \eqref{I1estimate} we have $I_1\asymp I_1^{(1)}+I_1^{(2)}$, where
\begin{align*}
I_1^{(1)}=&\int_{0}^{\xi/\eta} \frac{\xi s^{k-1}ds}{(\xi+\eta s)(\zeta+\eta s)^{k}}\asymp 
\int_{0}^{\xi/\eta} \frac{s^{k-1}ds}{(\zeta+\eta s)^{k}}\\
=&\eta^{-k}\left(\int_{0}^{\zeta/\eta} \frac{s^{k-1}ds}{(\zeta/\eta+ s)^{k}}+\int_{\zeta/\eta}^{\xi/\eta} \frac{s^{k-1}ds}{(\zeta/\eta+ s)^{k}}\right)\\
\asymp &\eta^{-k}\left((\eta/\zeta)^k\int_{0}^{\zeta/\eta} s^{k-1}ds+\int_{\zeta/\eta}^{\xi/\eta} \frac{ds}{s}\right)\\
=&\eta^{-k}\left(1/k+\log(\xi/\eta)-\log(\zeta/\eta)\right)\\
\asymp &\eta^{-k}\left(1\vee\log(\xi/\zeta)\right)\asymp (\zeta+\eta)^{-k}\left(1\vee\log(\xi/\zeta)\right),
\end{align*}
and
\begin{align*}
I_1^{(2)}=&\int_{\xi/\eta}^1 \frac{\xi s^{k-1}ds}{(\xi+\eta s)(\zeta+\eta s)^{k}}\asymp
\int_{\xi/\eta}^1 \frac{\xi ds}{\eta^{k+1}s^2}=(\eta-\xi)/\eta^{k+1}\\
\leq& \eta^{-k}\asymp(\zeta+\eta)^{-k}\leq(\zeta+\eta)^{-k}\left(1\vee\log(\xi/\zeta)\right).
\end{align*}
Hence
$$
I_1\asymp I_1^{(1)}+I_1^{(2)}\asymp I_1^{(1)}\asymp(\zeta+\eta)^{-k}\left(1\vee\log(\xi/\zeta)\right).
$$
This and \eqref{I2estimate} give $I_1+I_2\asymp I_1$.
\end{enumerate}
Altogether, for $\xi>\zeta$ we have
$$
G_k(x,y)\asymp\frac{1}{(\zeta+\eta)^k}\left[1\vee\log\left(\frac{\eta\wedge\xi}{\zeta}\right)\right],
$$
and \eqref{GreenEasyCase} with $d=2$ otherwise. Hence \eqref{GreenD2right} follows.

It remains to estimate $I_1$ for $d\geq3$.
\begin{enumerate}
\item[(a)] $\xi\geq \eta$. By \eqref{I1estimate} and the estimates from the proof of Lemma \ref{lem:DunklNewtonEst1} we have
$$
I_1\asymp\int_{0}^1 \frac{s^{k-1}ds}{(\zeta+\eta s)^{k+d/2-1}}\asymp
(\zeta+\eta)^{-k}\zeta^{1-d/2}.
$$
Combining this with \eqref{I2estimate} give $I_1+I_2\asymp I_1$.
\item[(b)] $\zeta<\xi<\eta$. By \eqref{I1estimate} we have $I_1\asymp I_1^{(1)}+I_1^{(2)}$, 
where
\begin{align*}
I_1^{(1)}=&\int_{0}^{\xi/\eta} \frac{\xi s^{k-1}ds}{(\xi+\eta s)(\zeta+\eta s)^{k+d/2-1}}\asymp 
\int_{0}^{\xi/\eta} \frac{s^{k-1}ds}{(\zeta+\eta s)^{k+d/2-1}}\\
=&\eta^{1-k-d/2}\left(\int_{0}^{\zeta/\eta} \frac{s^{k-1}ds}{(\zeta/\eta+ s)^{k+d/2-1}}+
\int_{\zeta/\eta}^{\xi/\eta} \frac{s^{k-1}ds}{(\zeta/\eta+ s)^{k+d/2-1}}\right)\\
\asymp &\eta^{1-k-d/2}\left((\eta/\zeta)^{k+d/2-1}\int_{0}^{\zeta/\eta} s^{k-1}ds+
\int_{\zeta/\eta}^{\xi/\eta} s^{-d/2}ds\right)\\
=&\eta^{1-k-d/2}\left(\frac{(\eta/\zeta)^{d/2-1}}{k}+\frac{2}{d-2}
\left[(\eta/\zeta)^{d/2-1}-(\eta/\xi)^{d/2-1}\right]\right)\\
\asymp &\eta^{1-k-d/2}(\eta/\zeta)^{d/2-1}
\asymp (\zeta+\eta)^{-k}\zeta^{1-d/2},
\end{align*}
and
\begin{align*}
I_1^{(2)}=&\int_{\xi/\eta}^1 \frac{\xi s^{k-1}ds}{(\xi+\eta s)(\zeta+\eta s)^{k+d/2-1}}\asymp
\int_{\xi/\eta}^1 \frac{\xi ds}{\eta^{k+d/2}s^{d/2+1}}\\
=& \frac{2\xi}{d\eta^{k+d/2}}\left[(\eta/\xi)^{d/2}-1\right]\leq\eta^{-k}\xi^{1-d/2}\asymp 
(\zeta+\eta)^{-k}\zeta^{1-d/2}.
\end{align*}
Hence
$$
I_1=I_1^{(1)}+I_1^{(2)}\asymp I_1^{(1)}\asymp(\zeta+\eta)^{-k}\zeta^{1-d/2}.
$$
This and \eqref{I2estimate} give $I_1+I_2\asymp I_1$.
\end{enumerate}
Altogether, $G_k(x,y)\asymp(\zeta+\eta)^{-k}\zeta^{1-d/2}$ for $\xi>\zeta$, and \eqref{GreenEasyCase} otherwise. This proves \eqref{GreenD3right}.
\end{proof}

\vspace{5mm}

\begin{lem}\label{lem:GreenEst1left}
The two-sided bound of $G_k(x,y)$ on $\left\{(x,y)\in\BB\times\BB:x_1y_1<0\right\}$ is the following.
\begin{enumerate}
\item[1.] If $d=2$, then
\begin{align}\label{GreenD2left}
G_k(x,y)\asymp\frac{1}{|x-y|^{2k}}&\left(1\wedge\frac{\delta(x)\delta(y)}{|x-y|^{2}}\right)\\
&\times\left[1\vee\log\left(\frac{|x_1y_1|}{\delta(x)\delta(y)\vee|x-\sigma y|^2}\right)\right].\nonumber
\end{align}
\item[2.] If $d=3$, then
\begin{align}\label{GreenD3left}
G_k(x,y)\asymp\frac{1}{|x-y|^{2k+1}}\left(1\wedge\frac{\sqrt{\delta(x)\delta(y)}}{|x-\sigma y|}\right)
\left(1\wedge\frac{\sqrt{\delta(x)\delta(y)}}{|x-y|}\right).
\end{align}
\item[3.] If $d=4$, then
\begin{align}\label{GreenD4left}
G_k(x,y)\asymp\frac{1}{|x-y|^{2k+2}}&\left(1\wedge\frac{\delta(x)\delta(y)}{|x-\sigma y|^2}\right)\\
&\times\left[1\vee\log\left(\frac{|x_1y_1|\wedge\delta(x)\delta(y)}{|x-\sigma y|^2}\right)\right].\nonumber
\end{align}
\item[4.] If $d\geq5$, then
\begin{align}\label{GreenD5left}
G_k(x,y)&\asymp\frac{1}{|x-y|^{2k+2}|x-\sigma y|^{d-4}}\left(1\wedge\frac{\delta(x)\delta(y)}{|x-\sigma y|^2}\right).
\end{align}
\end{enumerate}
\end{lem}

\vspace{5mm}

\begin{proof}
Denote $\zeta=|x-\sigma y|^{2}$, $\eta=|x_1y_1|$, and $\xi=\delta(x)\delta(y)$. By Theorem \ref{thm:GreenEst}, \eqref{IntertwinerRank1}, and \eqref{NormId} we have
\begin{equation}\label{Rank1GreenAsympLeft}
G_k(x,y)\asymp\int_{-1}^1 \frac{\xi(1-t)^{k-1}(1+t)^kdt}{(\xi+\zeta+\eta(1+t))(\zeta+\eta(1+t))^{k+d/2-1}}.
\end{equation}
Assume first $\xi\leq\zeta$. Then by \eqref{Rank1GreenAsympLeft},
$$
G_k(x,y)\asymp\int_{-1}^1 \frac{\xi(1-t)^{k-1}(1+t)^kdt}{(\zeta+\eta(1+t))^{k+d/2}}.
$$
Let $d=2$. Using the estimate derived for \eqref{NewtonPotentialRank1Left} with $d'=4$ instead of $d$, we get by \eqref{DunklNewtonD4left} that
\begin{equation}\label{GreenEasyCaseLeft1}
G_k(x,y)\asymp\frac{\xi}{(\zeta+\eta)^{k+1}}
\left(1\vee\log\frac{\eta}{\zeta}\right).
\end{equation}
If $d\geq3$, then \eqref{DunklNewtonD5left} with $d'=d+2$ instead of $d$ gives
\begin{equation}\label{GreenEasyCaseLeft2}
G_k(x,y)\asymp\frac{\xi}{(\zeta+\eta)^{k+1}\zeta^{d/2-1}}.
\end{equation}
Assume $\xi>\zeta$. Using \eqref{Rank1GreenAsympLeft} and substituting $s=t+1$ we get
$$
G_k(x,y)\asymp
\int_{0}^2 \frac{\xi(2-s)^{k-1}s^kds}{(\xi+\eta s)(\zeta+\eta s)^{k+d/2-1}}=I_1+I_2,
$$
where
\begin{equation}\label{I1estimateLeft}
I_1\asymp\int_{0}^1 \frac{\xi s^{k}ds}{(\xi+\eta s)(\zeta+\eta s)^{k+d/2-1}},
\end{equation}
and
\begin{equation}\label{I2estimateLeft}
I_2\asymp\int_{1}^2 \frac{\xi(2-s)^{k-1}ds}{(\xi+\eta s)(\zeta+\eta s)^{k+d/2-1}}\asymp \frac{\xi}{(\xi+\eta)(\zeta+\eta)^{k+d/2-1}}.
\end{equation}
In order to estimate $I_1$ we need to consider several cases. 
\begin{enumerate}
\item[(a)] $\xi\geq\eta$. Then the estimate depends on the dimension as follows.
\begin{itemize}
\item[(i)] $2\leq d\leq3$. \eqref{I1estimateLeft} and the estimates derived in the proof of Lemma \ref{lem:DunklNewtonEst2} give
$$
I_1\asymp\int_{0}^1 \frac{s^{k}ds}{(\zeta+\eta s)^{k+d/2-1}}\asymp
\frac{1}{(\zeta+\eta)^{k+d/2-1}}.
$$
In view of \eqref{I2estimateLeft}, we also have $I_1+I_2\asymp I_1$.
\item[(ii)] $d=4$. The same arguments as above give
\begin{equation}
I_1\asymp\int_{0}^1 \frac{s^{k}ds}{(\zeta+\eta s)^{k+1}}\asymp\frac{1}{(\zeta+\eta)^{k+1}}\left(1\vee\log\frac{\eta}{\zeta}\right),\nonumber
\end{equation}
and $I_1+I_2\asymp I_1$.
\item[(iii)] $d\geq5$. We get
$$
I_1+I_2\asymp I_1\asymp\frac{1}{(\zeta+\eta)^{k+1}\zeta^{d/2-2}}.
$$
\end{itemize}
\item[(b)] $\xi<\eta$. Then $\zeta<\xi<\eta$.
By \eqref{I1estimateLeft}, for any $d\geq2$ we have $I_1\asymp I_1^{(1)}+I_1^{(2)}$, where
\begin{align}
I_1^{(1)}&=\int_{0}^{\xi/\eta} \frac{\xi s^{k}ds}{(\xi+\eta s)(\zeta+\eta s)^{k+d/2-1}}\asymp
\int_{0}^{\xi/\eta}\frac{s^{k}ds}{(\zeta+\eta s)^{k+d/2-1}}\nonumber\\
&=\eta^{1-k-d/2}\left(\int_{0}^{\zeta/\eta}\frac{s^{k}ds}{(\zeta/\eta+ s)^{k+d/2-1}}+
\int_{\zeta/\eta}^{\xi/\eta}\frac{s^{k}ds}{(\zeta/\eta+ s)^{k+d/2-1}}\right)\nonumber\\
&\asymp\eta^{1-k-d/2}\left((\eta/\zeta)^{k+d/2-1}\int_{0}^{\zeta/\eta}s^{k}ds+
\int_{\zeta/\eta}^{\xi/\eta}s^{1-d/2}ds\right),\label{I11estimate}
\end{align}
and 
\begin{align}\label{I12estimate}
I_1^{(2)}&=\int_{\xi/\eta}^1 \frac{\xi s^{k}ds}{(\xi+\eta s)(\zeta+\eta s)^{k+d/2-1}}\asymp
\frac{\xi}{\eta^{k+d/2}}\int_{\xi/\eta}^1s^{-d/2}ds.
\end{align}
At this point we need to consider different values of $d$ separately.
\begin{itemize}
\item[(i)] $d=2$. By \eqref{I11estimate} we have
$$
I_1^{(1)}\asymp\eta^{-k}\left(\frac{\zeta}{(k+1)\eta}+\frac{\xi-\zeta}{\eta}\right)\asymp\frac{\xi}{\eta^{k+1}},
$$
and by \eqref{I12estimate},
$$
I_1^{(2)}\asymp\frac{\xi}{\eta^{k+1}}\log\frac{\eta}{\xi}.
$$
Therefore,
$$
I_1\asymp I_1^{(1)}+I_1^{(2)}\asymp\frac{\xi}{\eta^{k+1}}\left(1\vee\log\frac{\eta}{\xi}\right).
$$
Since $\zeta<\xi<\eta$, the last estimate of $I_1$ and \eqref{I2estimateLeft} give
$$
I_1+I_2\asymp I_1\asymp\frac{\xi}{(\zeta+\eta)^{k+1}}\left(1\vee\log\frac{\eta}{\xi}\right).
$$
Combining this with (a)(i) and \eqref{GreenEasyCaseLeft1} we get \eqref{GreenD2left}.
\item[(ii)] $d=3$. Then \eqref{I11estimate} gives
\begin{align*}
I_1^{(1)}&\asymp\frac{1}{\eta^{k+1/2}}\left[\frac{1}{k+1}\sqrt{\zeta/\eta}+2\left(\sqrt{\xi/\eta}-\sqrt{\zeta/\eta}\right)\right]\asymp\frac{\sqrt{\xi}}{\eta^{k+1}},
\end{align*}
and by \eqref{I12estimate},
$$
I_1^{(2)}\asymp\frac{\xi}{\eta^{k+3/2}}\left(\sqrt{\eta/\xi}-1\right)\leq\frac{\sqrt{\xi}}{\eta^{k+1}}\asymp\frac{\sqrt{\xi}}{(\zeta+\eta)^{k+1}}.
$$
Therefore, $I_1\asymp I_1^{(1)}+I_1^{(2)}\asymp I_1^{(1)}$. Furthermore, by \eqref{I2estimateLeft} we have
\begin{align*}
I_2\leq\left(\frac{\xi}{\xi+\eta}\right)^{1/2}\frac{1}{(\zeta+\eta)^{k+1/2}}\leq\frac{\sqrt{\xi}}{(\zeta+\eta)^{k+1}}.
\end{align*}
Hence 
$$
I_1+I_2\asymp I_1\asymp\frac{\sqrt{\xi}}{(\zeta+\eta)^{k+1}}.
$$
This, (a)(i) and \eqref{GreenEasyCaseLeft2} give \eqref{GreenD3left}.
\item[(iii)] $d=4$. By \eqref{I11estimate},
\begin{align*}
I_1^{(1)}\asymp\eta^{-k-1}\left(\frac{1}{k+1}+\log\frac{\xi}{\zeta}\right)\asymp
\frac{1}{(\zeta+\eta)^{k+1}}\left(1\vee\log\frac{\xi}{\zeta}\right),
\end{align*}
and by \eqref{I12estimate},
$$
I_1^{(2)}\asymp\frac{\xi}{\eta^{k+2}}\left(\frac{\eta}{\xi}-1\right)\leq\frac{1}{\eta^{k+1}}\asymp\frac{1}{(\zeta+\eta)^{k+1}}.
$$
Hence $I_1\asymp I_1^{(1)}+I_1^{(2)}\asymp I_1^{(1)}$. Combining this with \eqref{I2estimateLeft} we get 
$$
I_1+I_2\asymp I_1\asymp\frac{1}{(\zeta+\eta)^{k+1}}\left(1\vee\log\frac{\xi}{\zeta}\right).
$$
The last estimate, (a)(ii) and \eqref{GreenEasyCaseLeft2} give \eqref{GreenD4left}.
\item[(iv)] $d\geq5$. By \eqref{I11estimate},
\begin{align*}
I_1^{(1)}\asymp\frac{1}{\eta^{k+d/2-1}}\left[2\left(\frac{\eta}{\zeta}\right)^{d/2-2}-\left(\frac{\eta}{\xi}\right)^{d/2-2}\right]\asymp\frac{1}{\eta^{k+1}\zeta^{d/2-2}},
\end{align*}
and by \eqref{I12estimate},
$$
I_1^{(2)}\asymp\frac{\xi}{\eta^{k+d/2}}\left[\left(\frac{\eta}{\xi}\right)^{d/2-1}-1\right]\leq\frac{1}{\eta^{k+1}\xi^{d/2-2}}\leq\frac{1}{\eta^{k+1}\zeta^{d/2-2}}.
$$
It follows that $I_1\asymp I_1^{(1)}+I_1^{(2)}\asymp I_1^{(1)}$. Furthermore, by \eqref{I2estimateLeft},
$$
I_2\leq\frac{1}{(\zeta+\eta)^{k+d/2-1}}\leq\frac{1}{(\zeta+\eta)^{k+1}\zeta^{d/2-2}}.
$$
Hence
$$
I_1+I_2\asymp I_1\asymp\frac{1}{\eta^{k+1}\zeta^{d/2-2}}\asymp\frac{1}{(\zeta+\eta)^{k+1}\zeta^{d/2-2}}.
$$
The same estimate holds also in (a)(iii). Combining this with \eqref{GreenEasyCaseLeft2} we obtain \eqref{GreenD5left}.
\end{itemize}
\end{enumerate}
\end{proof}

By \eqref{PoissonRoesler} and \eqref{IntertwinerRank1}, the Poisson kernel in the rank one case in $\RR^d$ can be written as
\begin{equation}\label{eq:PoissonRank1}
P_k(x,y)=c_k \int_{-1}^1 \frac{(1-|x|^2)(1-t)^{k-1}(1+t)^k
dt}{(|x|^2+1-2(tx_1y_1+x_2y_2+...+x_dy_d))^{k+d/2}}.
\end{equation}
As a consequence of the two-sided bounds of the Newton kernel obtained in Theorem~\ref{thm:NewtonEst} we get the following two-sided estimates of $P_k(x,y)$.

\begin{cor}\label{thm:Poisson}
Let $\Phi(x,y):=|x-y|\vee|x-\sigma y|$. The two-sided bound of $P_k(x,y)$ on $\BB\times\s$ is the following.
\begin{enumerate}
\item[1.] If $d=2$, then
\begin{equation}\label{est:Poissond2}
P_k(x,y)\asymp \frac{1-|x|^2}{\Phi(x,y)^{2k}|x-y|^2}\left[1\vee\log\left(\frac{|x_1y_1|}{|x-\sigma y|^2}\right)\right].
\end{equation}
\item[2.] If $d\geq3$, then
\begin{equation}\label{est:Poissond3}
P_k(x,y)\asymp \frac{1-|x|^2}{\Phi(x,y)^{2k}|x-y|^2(|x-y|\wedge|x-\sigma y|)^{d-2}}.
\end{equation}
\end{enumerate}
\end{cor}

\begin{proof}
In view of the formulas \eqref{DunklNewtonR1} and \eqref{eq:PoissonRank1}, we can apply Theorem \ref{thm:NewtonEst} with $d'=d+2$ instead of $d$. Hence, \eqref{est:Poissond2} follows from \eqref{est:NewtonD4}, and \eqref{est:Poissond3} follows from \eqref{est:NewtonD5}.
\end{proof}

\begin{remark}
\rm
When $d=1$, the condition $k>1/2$ guarantees that $N_k(x,y)$ is well defined and finite, and hence also $G_k(x,y)$ and $P_k(x,y)$. Using the methods of this section one can derive the following two-sided estimates.
\begin{align*}
N_k(x,y)&\asymp(|x|+|y|)^{1-2k},\quad x,y\in\RR,\\
G_k(x,y)&\asymp\frac{\sqrt{\delta(x)\delta(y)}}{(|x|+|y|)^{2k-1}}\left(1\wedge\frac{\sqrt{\delta(x)\delta(y)}}{|x-y|}\right),\quad x,y\in(-1,1),\\
P_k(x,y)&\asymp 1-|x|,\quad x\in(-1,1),y\in\{-1,1\}.
\end{align*}
\end{remark}


\begin{remark}
\rm 
It is noteworthy that the explicit formulas for $N_k(x,y)$, $G_k(x,y)$ and $P_k(x,y)$ can be obtained in some particular cases, e.g., for $k\in\NN$ and $d\in 2\NN$ the integrands in formulas \eqref{DunklNewtonR1} and \eqref{eq:PoissonRank1} are  rational functions of $t$.
For instance, when $k=1$ and $d=2$ (i.e. for the root system $A_1$ in $\RR^2$), we can derive the following explicit expressions
\begin{align}\label{N1P1G1}
N_1(x,y)&=\frac{1}{4\pi}\left[\frac{|x-\sigma y|^2}{2x_1^2y_1^2}\log\left(\frac{|x-\sigma y|}{|x-y|}\right)-\frac{1}{x_1y_1}\right],\\
P_1(x,y)&=\frac{1-|x|^2}{4x_1^2y_1^2}\left[\frac{2x_1y_1}{|x-y|^2}
+\log\left(\frac{|x-y|}{|x-\sigma y|}\right)\right],\nonumber\\
G_1(x,y)&=\frac{|x-\sigma y|^2}{8\pi x_1^2y_1^2}\log\left(\frac{|x-\sigma y|}{|x-y|}\right)-\frac{|x|^2|x^*-\sigma y|^2}{8\pi x_1^2y_1^2}\log\left(\frac{|x^*-\sigma y|}{|x^*-y|}\right).\nonumber
\end{align}
\end{remark}

\begin{remark}{\bf $W$-radial case and applications to the Dyson Brownian Motion.}
\rm
The results of this paper can be applied to the $W$-invariant part of the Dunkl Laplacian,
$$\Delta^W_kf(x) = \Delta f(x) + \sum_{\alpha\in R_+} k(\alpha)\frac{\partial_\alpha f(x)}{\langle \alpha, x\rangle}.$$
Notice that for $k=1$ and $W=S_{d-1}$ this is just the generator of the $d$-dimensional Dyson Brownian motion.
 In fact, for all integral kernels $K(x,y)$ for $\Delta_k$ considered in the paper, the following formula holds
\begin{equation}\label{eq:radial}
K^W(x,y)=\sum_{g\in W}K(x,gy),
\end{equation}
where $K^W$ is the corresponding kernel for the operator  $\Delta^W_k$.\\
In the rank one case with $k=1$ and $d=2$, formulas \eqref{N1P1G1} and \eqref{eq:radial} give
\begin{align*}
N^W_1(x,y)&=\frac{1}{2\pi x_1y_1}\log\left(\frac{|x-\sigma y|}{|x-y|}\right),\\
P^W_1(x,y)&=\frac{2(1-|x|^2)}{|x-y|^2|x-\sigma y|^2},\nonumber\\
G^W_1(x,y)&=\frac{1}{2\pi x_1y_1}\log\left(\frac{|x^*-y||x-\sigma y|}{|x-y||x^*-\sigma y|}\right). \nonumber
\end{align*}
Furthermore, by multiplying the above formulas by $\omega_1(y)=y_1^2$ and going back to the initial form $A_1=\{\pm(e_1-e_2)\}$ with the standard basis vectors $e_1$, $e_2$ one obtains the Newton kernel, Poisson kernel and Green function of the unit ball in the setting of the potential theory of 2-dimensional Dyson Brownian motion:
\begin{align*}\label{Dyson}
N^{Dys}_1(x,y)&=\frac{1}{2\pi}
\frac{ \pi(y)}{\pi(x)}\log\left(\frac{|x-\sigma_\alpha y|}{|x-y|}\right),\\
P^{Dys}_1(x,y)&=\frac{2\pi(y)^2(1-|x|^2)}{|x-y|^2|x-\sigma_\alpha y|^2}, \nonumber\\
G^{Dys}_1(x,y)&=\frac{1}{2\pi}
\frac{ \pi(y)}{\pi(x)}
\log\left(\frac{|x^*-y||x-\sigma_\alpha y|}{|x-y||x^*-\sigma_\alpha y|}\right), \nonumber
\end{align*}
where  $x,y$ are in the positive Weyl chamber $ C^+=\{(z_1,z_2):\ z_1>z_2\}$,  $\pi(z)=z_1-z_2$, and $\sigma_\alpha(z_1,z_2)=(z_2,z_1)$.
\end{remark}

%



\end{document}